\documentclass{amsart}
\usepackage{hyperref}
\hypersetup{
  colorlinks,
  citecolor=blue,
 linkcolor=blue,
  urlcolor=blue}

\usepackage{amsmath,amssymb,amsthm}
\usepackage{comment,cleveref,enumerate}
\usepackage[all]{xy}

\crefname{equation}{equation}{equations}
\crefname{section}{Section}{Sections}
\crefname{chapter}{Chapter}{Chapters}

\newenvironment{manualconj}[2][Conjecture]{\begin{itshape}\begin{trivlist}
\item[\hskip \labelsep {\bfseries #1}\hskip \labelsep {\bfseries #2}]}{\end{trivlist}\end{itshape}}

\newtheorem{lem}{Lemma}[section]
\newtheorem{thm}[lem]{Theorem}
\newtheorem*{thm*}{Theorem}

\newtheorem{prop}[lem]{Proposition}

\newtheorem{conj}[lem]{Conjecture}

\theoremstyle{definition}
\newtheorem{dfn}[lem]{Definition}

\newtheorem{question}[lem]{Question}

\newtheorem{ex}[lem]{Example}
\newtheorem{rmk}[lem]{Remark}

\newcommand{\NN}{\mathbb{N}}
\newcommand{\CC}{\mathbb{C}}

\newcommand{\PP}{\mathbb{P}}

\newcommand{\Spec}{\operatorname{Spec}}
\newcommand{\ord}{\operatorname{ord}}
\newcommand{\Aut}{\operatorname{Aut}}   

\newcommand{\cha}{\operatorname{char}}

\newcommand{\Sym}{\operatorname{Sym}}

\newcommand{\e}{\'{e}}
\newcommand{\gds}{G_d^{sep}}

\newcommand{\tw}[1]{{\widetilde{#1}}}

\newcommand{\benum}{\begin{enumerate}[(1)]}
\newcommand{\eenum}{\end{enumerate}}

\numberwithin{equation}{section}
\usepackage{comment,cleveref,enumerate}

\begin{document}

\title{Families and moduli of covers with specified ramification}

\author{Ryan Eberhart}
\address{Department of Mathematics, The Pennsylvania State University, University Park, PA 16802}
\email{rde3@psu.edu}


\date{}

\dedicatory{}

\begin{abstract}
We study branched covers of curves with specified ramification points, under a notion of equivalence derived from linear series. In characteristic 0, no non-constant families of covers with fixed ramification points exist. In positive characteristic we formulate a necessary and sufficient condition for the existence of such a family. We unconditionally prove one direction of this conjecture, and by studying infinitesimal deformations show the other direction in characteristic 2 and 3.
\end{abstract}

\maketitle

\section{Introduction}\label{introsect}
This paper concerns branched covers of curves, under a notion of equivalence derived from considering either linear series on a fixed curve $X$ or finite index subfields of the function field of $X$. 

A \textit{linear series} on $X$ is a linear subspace of the global sections of a line bundle on $X$. Two sections $s_1$ and $s_2$ of a line bundle that have no common zeroes determine a map $X\rightarrow\PP^1_k$, taking $P\in X$ to $(s_1(P):s_2(P))$. Conversely, a map to $\PP^1_k$ yields two sections of a line bundle with no common zeroes, by pulling back generators for the sheaf $\mathcal{O}(1)$. Hence, in studying a linear series $V$ it is natural to investigate planes inside $V$. One can easily see that choosing a different basis for the plane corresponds to post-composing the map determined by $s_1$ and $s_2$ with a fractional linear transformation. Therefore, planes inside a linear series correspond to maps $X\rightarrow\PP^1$ up to post-composition with fractional linear transformations.

With this motivation, we will henceforth consider two branched covering maps $f_i:X \to Y_i$ from a fixed source $X$ \textit{equivalent} if there is a commutative diagram:

$$\xymatrix{
& X\ar[ld]_{f_1}\ar[rd]^{f_2} & \\
Y_1\ar[rr]^\cong & & Y_2
}$$
For an explanation of the correspondence between covers with source $X$ under this equivalence and subfields of $\kappa(X)$, see \cite[Lemma 4.2]{mygal}. This notion of equivalence is distinct from the equivalence used when considering covers of a fixed target curve. A natural question is the following:

\begin{question}\label{quest}
Let $X$ be a smooth proper curve over an algebraically closed field and $S$ a finite set of points on $X$. Under what conditions does there exist a non-constant family of covers with source $X$ of fixed degree and ramification locus $S$?
\end{question}

In \cite{eh}, it is shown that such a family \textit{never} exists when the source is $\PP^1_\CC$. One can show more generally that the same holds for any curve over the complex numbers. In positive characteristic, based upon results concerning covers of a fixed target curve rather than a fixed source, one would expect no such tame families to exist. However, the following example illustrates that this is not the case:

\begin{ex}\label{tameex} (\cite[Example 5.6]{ossp})
Let $k$ be a field of characteristic $p>2$ and consider the family of covers $\PP^1_k\rightarrow \PP^1_k$ given by $y=x^{p+2}+tx^p+x$ with parameter $t\in k$. For every value of $t$, the cover is tamely ramified at $\infty$ and the $(p+1)^{st}$ roots of $-1/2$ and is \e tale elsewhere. When fixing the source $\PP^1_k$, no distinct values of $t$ produce covers which are equivalent up to an automorphism of the target $\PP^1_k$.
\end{ex}

Many other examples of non-constant tame families with fixed ramification indices at fixed points exist, but there is currently no conjectural necessary and sufficient condition for when such families exist. If instead of fixing the ramification indices we fix a related notion, the differential lengths (see \cref{difflensect} for the definition), we have the following conjecture:

\begin{manualconj}{\ref{mainc}.}
Let $k$ be an algebraically closed field of characteristic $p>0$, $d$ a positive integer, $S=\{P_1,...,P_n\}$ a set of points on $\PP^1_k$, and $l_1,...,l_n$ positive integers. Suppose there exists a degree $d$ cover $f:\PP^1_k\rightarrow\PP^1_k$ with ramification locus $S$ such that the differential length of $f$ at each $P_i$ is $l_i$. Then there exists a non-constant family of degree $d$  covers $\PP^1_k\rightarrow\PP^1_k$ with ramification locus $S$ and differential length $l_i$ at each $P_i$ if and only if $l_i\geq p$ for at least one $i$.
\end{manualconj}

In this article, we make progress towards resolving this conjecture. \cref{halfofmaincprop} proves one half of this conjecture. \cref{threerpsprop} shows the other half holds when there are at most three ramification points, and \cref{mainconjchar3thm} shows it holds in characteristic 2 and 3.

\medskip

\textit{Structure of the paper:} In \cref{difflensect}, we introduce the numerical ramification data we wish to fix, the differential lengths. The main conjecture concerning the existence of a non-constant family with fixed differential lengths is found in \cref{mainconjsect}.  The framework for studying families with fixed differential lengths is provided in \cref{modspsect}. As an application of this framework, the other half of \cref{mainc} is proven in characteristic 2 and 3 in \cref{mainconjchar3sect}, which is the main result of this paper.

\textit{Terminology and Conventions:} A morphism $f:X\rightarrow Y$ of smooth curves over an algebraically closed field is a \textit{branched cover} (sometimes shortened to \textit{cover}) if it is finite and generically \e tale. The set of points of $X$ at which $f$ is not \e tale is the \textit{ramification locus} of $f$. The Grassmannian of $n$-planes in a vector space $V$ will be denoted by $Gr(n,V)$. To avoid awkward phrasing, we always assume that families are over a connected base scheme.

Except for \cref{mainconjchar3sect}, the content of this paper is adapted from a portion of the author's PhD thesis at the University of Pennsylvania, under the direction of David Harbater.

\section{Differential lengths}\label{difflensect}
Let $f:X\rightarrow Y$ be a cover of smooth proper curves over an algebraically closed field. The sheaf $\Omega_{X/Y}$ of relative differentials of $f$ is torsion with support equal to the ramification locus of $f$. For $P\in X$ a closed point, the \textit{differential length} of $f$ at $P$, $l_P$, is the length of $(\Omega_{X/Y})_P$ as an $\mathcal{O}_{X,P}$ module. If $f$ is tamely ramified at $P$ then $l_P$ is equal to the ramification index minus one; if $f$ is wildly ramified at $P$ then $l_P$ is strictly larger than that by \cite[Chapter III Proposition 2.2]{hart}. The \textit{discriminant divisor} of $f$ is the Weil divisor
$Disc(f)=\sum_{P\in X} l_P\cdot P$.

Consider the case where $X$ and $Y$ are projective lines and $f:X\rightarrow Y$ is a degree $d$ cover. Choose a coordinate $x$ on $X$, let $\infty$ denote the unique pole of $x$, and let $U=X\setminus\{\infty\}$. The divisor $Disc(f)$ is principal when restricted to $U$, and is generated by a unique monic polynomial in $x$, which we denote by $disc_x(f)$. When the choice of coordinate is fixed, we may refer without qualification to the \textit{discriminant of $f$} and instead write $disc(f)$.

By the Riemann-Hurwitz formula, we have that
$$\sum_{P\in X} l_P=2d-2.$$
Therefore $l_\infty$ can be determined from the differential lengths of the points in $U$, which is discernible from $disc(f)$. Thus $Disc(f)$ and $disc(f)$ encode equivalent data. After choosing coordinates on both projective lines, we have an effective method for calculating $disc(f)$, which can be verified by calculating the order of vanishing of the pullback of a local generator for the sheaf of differentials:

\begin{prop}\label{discprop}
Let $k$ be an algebraically closed field and $f:\PP^1_k\rightarrow \PP^1_k$ a degree $d$ cover. Choose coordinates $x$ and $y$ on the source and target $\PP^1_k$ respectively. In these coordinates write $f$ as a rational function $y=g(x)/h(x)$ with $g(x)$ and $h(x)$ coprime. The discriminant $disc_x(f)$ is the unique monic polynomial which is a scalar multiple of $h(x)g'(x)-g(x)h'(x)$.
\end{prop}

\section{Main conjecture}\label{mainconjsect}
In formulating the definition of a family, the essential point we wish to capture is that in our notion of equivalence the source $X$ is fixed but the target is allowed to vary.

\begin{dfn}\label{famdfn}
Let $k$ be an algebraically closed field, $X$ a smooth proper curve over $k$, and $T$ a $k$-scheme. A \textit{family of degree $d$ covers with source X over $T$} is a $T$-morphism $f:X_T\rightarrow \mathcal{Y}$ where $\mathcal{Y}$ is a flat $T$-scheme and the geometric fibers of $f$ are degree $d$ covers between smooth proper curves. A family is \textit{constant} if the fibers over all $k$-points of $T$ are equivalent.
\end{dfn}

With the differential length viewpoint from \cref{difflensect}, we can state our conjecture concerning the existence of a family:

\begin{conj}\label{mainc}
Let $k$ be an algebraically closed field of characteristic $p>0$, $d$ a positive integer, $S=\{P_1,...,P_n\}$ a set of points on $\PP^1_k$, and $l_1,...,l_n$ positive integers. Suppose there exists a degree $d$ cover $f:\PP^1_k\rightarrow\PP^1_k$ with ramification locus $S$ such that the differential length of $f$ at each $P_i$ is $l_i$. Then there exists a non-constant family of degree $d$  covers $\PP^1_k\rightarrow\PP^1_k$ with ramification locus $S$ and differential length $l_i$ at each $P_i$ if and only if $l_i\geq p$ for at least one $i$.
\end{conj}

This conjecture implies an affirmative answer to Question 8.4 in \cite{ossp}. Showing one half of \cref{mainc} is not difficult:

\begin{prop}\label{halfofmaincprop}
Let $k$ be an algebraically closed field of characteristic $p>0$, $d$ a positive integer, $S=\{P_1,...,P_n\}$ a set of points on $\PP^1_k$, and $l_1,...,l_n$ positive integers such that $l_i\geq p$ for at least one $i$. Suppose there exists a degree $d$ cover $f:\PP^1_k\rightarrow\PP^1_k$ with ramification locus $S$ such that the differential length of $f$ at each $P_i$ is $l_i$. Then there exists a non-constant family of degree $d$ covers $\PP^1_k\rightarrow\PP^1_k$ with ramification locus $S$ and differential length $l_i$ at each $P_i$.
\end{prop}

\begin{proof}
Choose coordinates so that $l_\infty\geq p$ and $\infty$ is a fixed point. In these coordinates represent $f$ by a rational function $g(x)/h(x)$ with $g(x)$ and $h(x)$ coprime. Consider the family of covers
$$f_t(x)=\frac{g(x)}{h(x)}+t\cdot x^p=\frac{g(x)+t\cdot x^ph(x)}{h(x)}$$
with parameter $t\in k$. By our assumptions on $f$, each $f_t$ is written in lowest terms and is a degree $d$ cover. One checks directly that $f_t$ and $f_{t'}$ are equivalent only if $t=t'$. The discriminant of $f_t$ is $h(x)g'(x)-g(x)h'(x)$, independent of $t$. This implies that the given family has the desired properties.
\end{proof}

\begin{rmk}
We do not claim that the ramification indices in the above family remain fixed. It is often the case in producing such a family that the differential lengths are fixed but the ramification indices are not because the covers in the family vary between being wildly and tamely ramified, as the following simple example illustrates. The ability to more easily control the differential lengths in a family is an important reason to consider families with fixed differential lengths instead of fixed ramification indices.
\end{rmk}

\begin{ex}
Let $k$ be an algebraically closed field of characteristic $p>0$. Consider the cover $\PP^1_k\rightarrow\PP^1_k$ given by $y=x^{p+1}$, which has ramification index $p+1$ at 0 and $\infty$. Following the procedure of \cref{halfofmaincprop}, we construct the family of covers $y=x^{p+1}+t\cdot x^p$, which we have shown has fixed differential lengths. However, if $t\neq 0$, the ramification index at 0 is $p$. One can show that there is \textit{no} non-constant family of degree $p+1$ covers $\PP^1_k\rightarrow\PP^1_k$ with fixed ramification indices $e_0,e_\infty=p+1$, and hence \cref{mainc} would be false if we instead required fixed ramification indices.
\end{ex}

As evidence for the other direction of \cref{mainc}, when $d<p$ it holds by \cite[Corollary 3.2]{ossp}. When all the $l_i$ are even and less than $p$ it holds by \cite[Theorem 5.3]{ossm}. After introducing the necessary machinery in \cref{modspsect}, we will show this direction holds when $\cha(k)=2,3$ in \cref{mainconjchar3sect}. The following shows it holds when there are at most three ramification points:

\begin{prop}\label{threerpsprop}
Let $k$ be an algebraically closed field of characteristic $p>0$, $d$ a positive integer, $S=\{P_1,...,P_n\}$ a set of points on $\PP^1_k$, and $l_1,...,l_n$ positive integers. If $n\leq 3$ and $l_i<p$ for every $i$, then there are finitely many equivalence classes of degree $d$ covers $\PP^1_k\rightarrow \PP^1_k$ with ramification locus $S$ and differential length $l_i$ at each $P_i$. Furthermore, this implies that \cref{mainc} holds when $n\leq 3$.
\end{prop}

\begin{proof}
The conditions on the differential lengths imply that such a cover must have ramification index exactly $l_i+1$ at each $P_i$. By \cite[Theorem 3.3]{ossp}, for a generic choice of points $(Q_1,...,Q_n)$ there are finitely many equivalence classes of degree $d$ covers $\PP^1_k\rightarrow \PP^1_k$ with ramification index $l_i+1$ at each $Q_i$. Since $n\leq 3$, there is a bijection between the set of equivalence classes of covers with ramification points $(Q_1,...,Q_n)$ and $(P_1,...,P_n)$, given by pre-composing with an automorphism of $\PP^1_k$ sending one set of ramification points to the other. Hence there are finitely many equivalence classes with ramification points $(P_1,...,P_n)$ as well. By \cite[Theorem 2.2]{mygal}, this implies that there are no non-constant families of such covers.
\end{proof}

\section{Moduli space of covers with fixed differential lengths}\label{modspsect}

Let $X$ be a smooth proper curve over an algebraically closed field $k$. By \cite{ghbrill} there is a moduli space $G^1_d(X)$ of one dimensional linear series on $X$. Consider the case where $X$ is $\PP^1_k$. Since $\mathcal{O}(d)$ is the unique degree $d$ line bundle on $\PP^1_k$, $G^1_d(\PP^1_k)$ is isomorphic to $Gr(2,H^0(\mathcal{O}(d),\PP^1_k))$. Specifying the ramification index of a linear series at a point is equivalent to specifying that the linear series lies on a Schubert variety by \cite{eh}. By \cite[Proposition 2.4]{oss} a separable linear series is only ramified at finitely many points. Therefore the intersection of all these Schubert varieties is precisely the locus of inseparable linear series. The separable linear series then form an open subscheme of $G^1_d(\PP^1_k)$. This subscheme will be denoted by $\gds(\PP^1_k)$, or $\gds$ when the $\PP^1_k$ is understood.


Let $\PP^{2d-2}_k$ be the projectivization of $Poly_k(2d-2)$, the space of polynomials of degree at most $2d-2$ in $k[x]$. The discriminant from \cref{difflensect} defines a morphism $disc:\gds(\PP^1_k)\rightarrow \PP^{2d-2}_k$, since by the Riemann-Hurwitz formula the degree of the discriminant of a degree $d$ cover is at most $2d-2$.

\subsection{Fixing differential lengths in $\gds$}
We now examine subschemes of $\gds$ obtained by constraining the differential lengths. We begin with some preliminary definitions. Fix $d\in\NN$ and an effective Weil divisor $D$ on $\PP^1_k$ such that $\deg(D)=2d-2$. Let $\tw{X_D}\subseteq \gds$ be the set of closed points corresponding to linear series with discriminant divisor $D$. Fix positive integers $l_1,...,l_n$ such that $\sum l_i=2d-2$. Let $\tw{X_{(l_i)}}$ be the union of $\tw{X_D}$ over $D$ of the form $D=\sum_{i=1}^n l_i\cdot P_i$, where the $P_i\in\PP^1_k$ are distinct. The points of $\tw{X_{(l_i)}}$ correspond to linear series with fixed differential lengths, but where the ramification points are allowed to vary. We first show that $\tw{X_D}$ and $\tw{X_{(l_i)}}$ are the closed points of reasonable subschemes of $\gds$:

\begin{lem}\label{toplem}
There exists a unique reduced, closed subscheme of $\gds$, $X_D$, for which the set of closed points of $X_D$ is $\tw{X_D}$. Furthermore, there exists a unique reduced, locally closed subscheme of $\gds$, $X_{(l_i)}$, for which the set of closed points of $X_{(l_i)}$ is $\tw{X_{(l_i)}}$.
\end{lem}
\begin{proof}
We first prove the claim concerning $X_D$. Choose a coordinate function $x$ on $\PP^1_k$, and let $\alpha(x)$ be the polynomial discriminant corresponding to $D$. Let $X_D$ be the inverse image of $\alpha(x)$ under the map $disc:\gds\rightarrow \PP^{2d-2}_k$. It is clear that $X_D$ is closed, reduced, and the set of its closed points is $\tw{X_D}$, proving the existence portion of the claim.

To establish the uniqueness of $X_D$, suppose $X_D'$ is another subscheme of $\gds$ satisfying the same properties as $X_D$. By \cite[Chapter II Exc.\ 3.11(c)]{hart}, two reduced closed subschemes of $\gds$ which contain the same points must be isomorphic as subschemes, so there exists a (necessarily non-closed) point $P$ which is contained in either $X_D$ or $X_D'$, but not both. Since the closed points of a quasi-compact scheme are dense, $P$ lies in the closure $\tw{X_D}$, the set of closed points of both $X_D$ and $X_D'$. However, $X_D$ and $X_D'$ are both closed, so $P$ lies in both schemes, a contradiction. Hence $X_D$ must be unique.

For the remainder of this proof, use the same coordinate function $x$ on every copy of $\PP^1_k$. We now prove the claim concerning $X_{(l_i)}$. Let $\Sym^r(\PP^1_k)$ denote the $r$-fold symmetric product of $\PP^1_k$. Viewing $\PP^{2d-2}_k$ as the projectivization of $Poly_k(2d-2)$, we have an isomorphism $\varphi: \PP^{2d-2}_k\rightarrow \Sym^{2d-2}(\PP^1_k)$, which sends $f(x)=a\cdot\prod (x-c_i)^{d_i}$ to the point corresponding to $d_i$ copies of $(x-c_i)$ for each $i$ and $2d-2-\sum d_i$ copies of $\infty$.

Fix positive integers $l_1,...,l_n$ such that $\sum l_i=2d-2$. Let $\Delta_r:\PP^1_k\rightarrow \left(\PP^1_k\right)^r$ be the diagonal embedding and $\psi=(\Delta_{l_1},...,\Delta_{l_n}):\left(\PP^1_k\right)^n\rightarrow  \left(\PP^1_k\right)^{2d-2}$. Let $Z$ denote the image of $\left(\PP^1_k\right)^n$ under $\psi$. Since $\PP^1_k$ is separated, the image under each $\Delta_r$ is closed, and thus $Z$ is closed. Let $Z'$ be the $S_{2d-2}$-orbit of $Z$ and $\pi:\left(\PP^1_k\right)^{2d-2}\rightarrow \Sym^{2d-2}(\PP^1_k)$ the quotient map. Since $Z'$ is closed and $S_{2d-2}$-invariant, $\pi(Z')$ is closed. Endow $\pi(Z')$ with the reduced induced scheme structure.

Let $Y_{(l_1,...,l_n)}=disc^{-1}(\varphi^{-1}(\pi(Z')))$. By construction $\tw{X_{(l_i)}}\subseteq Y_{(l_i)}$. However $Y_{(l_1,...,l_n)}$ may additionally contain points corresponding to linear series with discriminant divisors of the form $\sum_{i=1}^n l_i\cdot P_i$, where the $P_i$ are not distinct. Each such point is contained in
$$Y_{(l_1,...,l_i+l_j,...,\widehat{l_j},...,l_n)}$$
for some choice of $i$ and $j$, where $\widehat{l_j}$ denotes that $l_j$ is omitted. Let $X_{(l_i)}$ be $Y_{(l_1,...,l_n)}$ with each such
$Y_{(l_1,...,l_i+l_j,...,\widehat{l_j},...,l_n)}$
removed. Since each removed scheme is closed and there are finitely many choices of indices to combine, $X_{(l_i)}$ is locally closed. After endowing $X_{(l_i)}$ with the induced open subscheme structure, this establishes the existence portion of the claim.

For the uniqueness assertion, suppose $X_{(l_i)}'$ is another subscheme of $\gds$ satisfying the same properties as $X_{(l_i)}$. Using the argument for the uniqueness of $X_D$, it must be the case that $X_{(l_i)}$ and $X_{(l_i)}'$ have the same closure, which we denote by $W$ and endow with the reduced induced scheme structure. Both $X_{(l_i)}$ and $X_{(l_i)}'$ are locally closed, so both are open in $W$. Since there is a unique induced subscheme structure on a fixed open set of $W$, $X_{(l_i)}$ and $X_{(l_i)}'$ cannot be set-theoretically equal. Hence there exists a point $P$ which is contained in either $X_{(l_i)}$ or $X_{(l_i)}'$, but not both. Without loss of generality assume $P\in X_{(l_i)}$. Since $W\setminus X_{(l_i)}'$ is closed and contains $P$, no point in the closure of $P$ is contained in $X_{(l_i)}'$. However the closure of $P$ inside $X_{(l_i)}$ is quasi-compact and thus contains a closed point, a contradiction to $X_{(l_i)}$ and $X_{(
l_i)}'$ containing the same closed points. Therefore $X_{(l_i)}$ is unique as claimed.
\end{proof}

\subsection{Deformations over local Artin rings}

We now characterize maps from local Artin rings with target $\gds$ in terms of the discriminant introduced in \cref{difflensect}. First, we require a definition of discriminant for rational functions over local Artin rings. Let $A$ be a local Artin ring, finite over a field $k$. Let $f=g(x)/h(x)$ where $g(x)$ and $h(x)$ are non-zero polynomials with coefficients in $A$. 
The \textit{discriminant} of $f$ is $h(x)g'(x)-g(x)h'(x)$.

\begin{prop}\label{formofartinringmapsprop}
Let $k$ be an algebraically closed field and $f:\PP^1_k\rightarrow\PP^1_k$ a degree $d$ cover. Choose coordinate functions on each copy of $\PP^1_k$ such that $f$ is unramified at $\infty$ and $\infty$ is a fixed point. In these coordinates represent $f$ by a rational function $y=g(x)/h(x)$, and if necessary adjust $y$ so that $g(x)$ and $h(x)$ are monic and $g(x)$ has no degree $d-1$ term.

Let $A=k[t_1,...,t_n]/I$ such that $\sqrt{I}=(t_1,...,t_n)$. Choose a basis $\{1,\tau_1,...,\tau_m\}$ of $A$ over $k$ such that each $\tau_i$ is a monomial. We have the following characterization of maps with source $\Spec(A)$:
\begin{enumerate}[(1)]
\item
Giving a map $\Spec(A)\rightarrow X_{Disc(f)}$ such that the unique closed point in $\Spec(A)$ maps to the point corresponding to $f$ is equivalent to choosing polynomials $g_1(x),...,g_m(x),h_1(x),...,h_m(x)\in k[x]$ of degree at most $d-2$ such that the discriminant of
\begin{equation}\label{defoverAeqn}
\frac{g_A(x)}{h_A(x)}=\frac{g(x)+g_1(x)\tau_1+\cdots +g_m(x)\tau_m}{h(x)+h_1(x)\tau_1+\cdots +h_m(x)\tau_m}
\end{equation}
is equal to $disc(f)$ modulo $I$.

\item Let $disc(f)=\prod (x-c_i)^{l_i}$ with the $c_i$ distinct. Giving a map $\Spec(A)\rightarrow X_{(l_i)}$ such that the unique closed point in $\Spec(A)$ maps to the point corresponding to $f$ is equivalent to choosing polynomials $g_1(x),...,g_m(x),h_1(x),...,h_m(x)\in k[x]$ of degree at most $d-2$ such that the discriminant of
\begin{equation*}
\frac{g_A(x)}{h_A(x)}=\frac{g(x)+g_1(x)\tau_1+\cdots +g_m(x)\tau_m}{h(x)+h_1(x)\tau_1+\cdots +h_m(x)\tau_m}
\end{equation*}
is equal to
$$\prod_{i=1}^N (x-c_i+d_i(t_1,...,t_n))^{l_i}$$
modulo $I$, where each $d_i$ contains no constant term.
\eenum
\end{prop}
\begin{proof}
Recall that $\gds$ is an open subscheme of $Gr(2,Poly_k(d))$. The planes which have a basis of the form $\{x^d+\alpha_{d-2}x^{d-2}+\cdots+\alpha_0,x^{d-1}+\beta_{d-2}x^{d-2}+\cdots+\beta_0\}$, where the $\alpha_i$ and $\beta_j$ are arbitrary constants, form an open subscheme $U$, which is isomorphic to $\Spec(R)$ where $R=k[\alpha_0,...,\alpha_{d-2},\beta_0,...,\beta_{d-2}]$. By our choice of coordinates, we have that

$$\frac{g(x)}{h(x)}=\frac{x^d+a_{d-2}x^{d-2}+\cdots+a_0}{x^{d-1}+b_{d-2}x^{d-2}+\cdots + b_0}.$$
Therefore the plane corresponding to $f$ is contained in $U$. Giving a map with source $\Spec(A)$ and target $X_{Disc(f)}$ or $X_{(l_i)}$ such that the unique closed point in $\Spec(A)$ maps to the point corresponding to $f$ is then equivalent to giving such a map with target $X_{Disc(f)}|_U$ or $X_{(l_i)}|_U$ respectively.

We first show (i), and begin with a map $\Spec(A)\rightarrow X_{Disc(f)}|_U$. By \cref{toplem}, $X_{Disc(f)}\subseteq\gds$ is closed. Therefore $X_{Disc(f)}|_U$ is isomorphic to $\Spec(R/K)$ for some ideal $K$. We now describe $K$. Let $$\tw{f}=\frac{\tw{g(x)}}{\tw{h(x)}}=\frac{x^d+\alpha_{d-2}x^{d-2}+\cdots+\alpha_0}{x^{d-1}+\beta_{d-2}x^{d-2}+\cdots+\beta_0}$$
be a cover corresponding to an arbitrary closed point in $U$. This point is contained in $X_{Disc(f)}$ precisely when
\begin{equation}\label{ftwiddleeqn}
\tw{h(x)}\tw{g'(x)}-\tw{g(x)}\tw{h'(x)}=disc(f).
\end{equation}
Since both sides are monic, equating powers of $x$ in \cref{ftwiddleeqn} yields $2d-2$ relations among the $\alpha_i$ and $\beta_j$ which generate the ideal $K$.

Having reduced to a map between affine schemes, we instead consider the corresponding ring homomorphism $\varphi: R/K\rightarrow A$. We must have that $\varphi(\alpha_i-a_i),\varphi(\beta_j-b_j)\in(t_1,...,t_n)$ for every $i$ and $j$ since the unique closed point of $\Spec(A)$ maps to the point corresponding to $f$. This implies that the image under $\varphi$ of each $(\alpha_i-a_i)$ has no constant term, and can therefore be written uniquely modulo $I$ as $\sum c_{i,l}\cdot \tau_l$ where each $c_{i,l}\in k$. Likewise, each $\varphi(\beta_j-b_j)$ can be written uniquely modulo $I$ as $\sum e_{j,l}\cdot \tau_l$ where each $e_{j,l}\in k$.

Write $g_i(x)=c_{d-2,i}x^{d-2}+\cdots+c_{0,i}$ and $h_j(x)=e_{d-2,j}x^{d-2}+\cdots+e_{0,j}$. The condition that $\varphi(K)\subseteq I$ is equivalent to the discriminant of
$$\frac{g(x)+g_1(x)\tau_1+\cdots +g_m(x)\tau_m}{h(x)+h_1(x)\tau_1+\cdots +h_m(x)\tau_m}$$
being equal to $disc(f)$ modulo $I$, as desired. Reversing the argument yields the other direction, proving (i).

The proof of (ii) proceeds similarly, with the caveat that $X_{(l_i)}$ is locally closed instead of closed. However, a map $\Spec(A)\rightarrow X_{(l_i)}$ such that the unique closed point in $\Spec(A)$ maps to the point corresponding to $f$ is equivalent to such a map with target the closure of $X_{(l_i)}\subseteq\gds$. Therefore we may follow the same process of describing the relations in the ideal corresponding to the closure of $X_{(l_i)}$ inside $U$ as was done for the ideal corresponding to $X_D$ inside $U$.
\end{proof}

\begin{rmk}
We refer to a rational function of the form in \cref{defoverAeqn} as a \textit{deformation of $f$ over $A$}.
\end{rmk}

\subsection{A differential operator in positive characteristic}

Let $k$ be an algebraically closed field of characteristic $p>0$ and consider $k(x)$ a vector space over $k(x^p)$ with basis $\{1,x,...,x^{p-1}\}$. Fix a non-zero rational function $f(x)\in k(x)$ and define the operator $T_f:k(x)\rightarrow k(x)$ by the formula $T_f(p(x))=p(x)f'(x)-f(x)p'(x)$. This operator will be utilized in the proof of \cref{mainconjchar3thm}, since $T_f(p(x))$ is the discriminant of the rational function $f(x)/p(x)$ when $f(x)$ and $p(x)$ are coprime polynomials.

Since differentiation is $k(x^p)$-linear, $T_f$ is a $k(x^p)$-linear operator. The $p$-fold composition of $T_f$, $(T_f)^p$, is in fact $k(x)$-linear. By a result in \cite{katz} (which in our situation is stated simply as \cite[Theorem 3.8]{cscurve}), the dimension of the kernel of $T_f$ as a $k(x^p)$-linear operator is equal to dimension of the kernel of $(T_f)^p$ as a $k(x)$-linear operator. Since $(T_f)^p$ operates on a 1 dimensional vector space and $T_f(f(x))=0$, this implies that the kernel of $T_f$ is 1 dimensional and the image of $T_f$ is $(p-1)$ dimensional.

\begin{prop}\label{laliftprop}
For the operator $T_f$ as defined above, if $\cha(k)=2$, then the image of $T_f$ is $k(x^2)\subseteq k(x)$ for every non-zero polynomial $f$. If $\cha(k)>2$, then every $(p-1)$ dimensional $k(x^p)$-vector subspace of $k(x)$ occurs as the image of $T_f$ for some polynomial $f$. Moreover, if $\cha(k)>2$ and the images of $T_f$ and $T_g$ are equal, then $f$ and $g$ are $k(x^p)$-multiples of each other.
\end{prop}
\begin{proof}
Fix a non-zero polynomial $f(x)\in k[x]$. First, consider the case where $\cha(k)=2$. Write $$f(x)=\sum_{i=0}^n (a_i\cdot x^{2i}+b_i\cdot x^{2i+1}).$$
By direct computation, we have that
$$T_f(1)=\sum_{i=0}^n b_i\cdot x^{2i},\ T_f(x)=\sum_{i=0}^n a_i\cdot x^{2i}.$$
Since $T_f(1)$ and $T_f(x)$ both lie in $k(x^2)$, the image of $T_f$ is indeed $k(x^2)$.

Consider the case where $\cha(k)>2$. Viewing $\mathbb{A}_{k(x^p)}^p$ as 
$\Spec(k(x^p)[1,...,x^{p-1}])$ and mapping $p(x)$ to the image of $T_p$, we have a map
$$\mathbb{A}_{k(x^p)}^p\setminus \{0\}\rightarrow Gr(p-1,k(x))\cong \PP^{p-1}_{k(x^p)}.$$
Since this map is constant on lines, it descends to a map $\varphi:\PP^{p-1}_{k(x^p)}\rightarrow \PP^{p-1}_{k(x^p)}$. A map between projective spaces of the same dimension must be either constant or surjective, and moreover in the latter case quasi-finite.

One checks directly in characteristic greater than 2 that $x^{p-1}$ lies in the image of $T_x$ but not $T_1$. This implies that $\varphi$ is non-constant and therefore surjective. That is, for every $(p-1)$ dimensional $k(x^p)$-vector subspace $V$ of $k(x)$, there is some rational function $f(x)$ such that the image of $T_f$ is $V$. Since multiplying $f(x)$ by $g(x)\in k[x^p]$ does not affect the image, for an appropriate choice of $g(x)$ we have that $f(x)g(x)$ is a polynomial such that the image of $T_{fg}$ is $V$, as desired.

Finally, suppose that $\cha(k)>2$ and the image of $T_f$ is equal to the image of $T_g$. If $f$ and $g$ correspond to the same point in $\PP^{p-1}_{k(x^p)}$, then they are $k(x^p)$-multiples of each other. Otherwise by varying $a$ and $b$, $af+bg$ forms a line on which $\varphi$ is constant, which contradicts $\varphi$ being quasi-finite.
\end{proof}

\section{Proof of conjecture in characteristic 2 and 3}\label{mainconjchar3sect}
\begin{thm}\label{mainconjchar3thm}
Let $k$ be an algebraically closed field of characteristic 2 or 3 and $D=\sum l_i\cdot P_i$ a divisor on $\PP^1_k$ such that the $P_i$ are distinct and $0 < l_i < p$ for each $i$. Then $X_D$ is empty or zero dimensional. This implies that \cref{mainc} holds when $\cha(k)=2,3$.
\end{thm}

\begin{proof}
First suppose that $\cha(k)=2$. No cover in characteristic 2 can have a differential length of 1 by \cite[Chapter III Proposition 2.2]{hart}. Therefore $X_D$ is empty if $D\neq 0$, or is a single point corresponding to the unique equivalence class of degree 1 covers if $D=0$.

Suppose that $\cha(k)=3$. Let $P \in X_D$ be a closed point and $f$ a degree $d$ cover in the equivalence class corresponding to $P$. We will show $X_D$ is zero dimensional at $P$ by showing that there are no non-trivial first order deformations of $f$. By \cref{formofartinringmapsprop}, after choosing appropriate coordinates so that $f=g(x)/h(x)$ is unramified at $\infty$ and $\infty$ is a fixed point, giving a first order deformation $\Spec(k[t]/(t^2))\rightarrow X_D$ mapping to $P$ is equivalent to choosing $g_1,h_1\in k[x]$ of degree at most $d-2$ such that the discriminant of
\begin{equation*}
f_t=\frac{g(x)+t\cdot g_1(x)}{h(x)+t\cdot h_1(x)}
\end{equation*}
is equal to $disc(f)$ modulo $(t^2)$.

We first deduce the form of $g_1$ and $h_1$. By \cref{laliftprop}, the $k(x^3)$-linear operators $T_g,T_h:k(x)\rightarrow k(x)$ have distinct 2 dimensional images, and hence the intersection of their images is 1 dimensional. Moreover, $g(x)$ and $h(x)$ span the kernels of $T_g$ and $T_h$ respectively. Since $T_g(h)=-T_h(g)$, we have that $T_g(p(x))=T_h(q(x))$ if and only if $p(x)=\beta h(x)$ and $q(x)=-\beta g(x)$ for some $\beta\in k(x^3)$. Therefore since the discriminant of $f_t$ has no $t$ term, it must be the case that $g_1(x)=\alpha h(x)+\beta g(x)$ and $h_1(x)=\gamma g(x)-\beta h(x)$, where $\alpha,\beta,\gamma\in k(x^3)$.

Let $\beta=\beta_1/\sigma$, where $\beta_1,\sigma\in k[x^3]$ are coprime. Suppose that $(x+a)^{3k}$ exactly divides $\sigma$ and let $v$ be the valuation on $k(x)$ associated to $(x+a)$. Since $0< l_i< 3$, it follows that $0 \leq v(g),v(h)<3$.  Since $g_1=\alpha h(x)+\beta g(x)$ is a polynomial, so $v(g_1)\geq 0$. However since $v(\beta g(x))< v(g_1)$, this implies that $v(\alpha h(x))=v(\beta g(x))$. Hence $(x+a)^{3k}$ divides the denominator of $\alpha$. Interchanging the roles of $\alpha,\beta,$ and $\gamma$ and applying the same argument to $h_1$, we deduce that the denominators of $\alpha,\beta,$ and $\gamma$ must be equal when written in lowest terms. Therefore, we may write $\alpha=\alpha_1/\sigma$ and $\gamma=\gamma_1/\sigma$, where $\alpha_1,\gamma_1\in k[x^3]$.

Since the degree of $g(x)$ and $h(x)$ differ by one and $\alpha_1,\beta_1,\gamma_1\in k[x^3]$, there can be no cancellation of highest order terms in $g_1=\alpha h(x)+\beta g(x)$ or $h_1=\gamma g(x)-\beta h(x)$. Therefore since the degree of $g_1$ and $h_1$ can be at most $d-2$, either $\alpha,\beta,\gamma=0$ or the degree of $\sigma$ is positive. To arrive at a contradiction, suppose that the former does not hold, and let $(x+a)^3$ be a divisor of $\sigma$. Since $g_1$ is a polynomial, we must have that $(x+a)^3$ divides $\alpha_1 h(x)+\beta_1 g(x)$ as well. The $k(x^3)$-linearity of $T_g$ then implies that $(x+a)^3$ divides $T_g(\alpha_1 h(x)+\beta_1 g(x))$, which is equal to $\alpha_1(h(x)g'(x)-g(x)h'(x))=\alpha_1\cdot disc(f)$. However $\alpha_1$ and $\sigma$ are coprime, so $(x+a)^3$ divides $disc(f)$. This is a contradiction to each $l_i$ being less than 3. Thus it must be the case that $\alpha,\beta,\gamma=0$, so that $P\in X_D$ has no non-trivial first order deformations. This implies that $X_D$ is zero dimensional at $P$. Since $X_D\subseteq \gds$ is closed, it then consists of finitely many closed points. By \cite[Theorem 2.2]{mygal}, this implies that no non-constant families of such covers exist.
\end{proof}

\begin{rmk}
The primary obstruction to carrying out the same argument in higher characteristic is that the intersection of the images of $T_g$ and $T_h$ is $(p-2)$ dimensional. Writing down the general form of $g_1$ and $h_1$ becomes difficult in this case, since the intersection of the images cannot be easily described in terms of $g$ and $h$.
\end{rmk}

\bibliography{mybib}{}
\bibliographystyle{amsalpha}
\end{document}